\documentclass{amsart}
\usepackage[T1]{fontenc}
\usepackage[utf8]{inputenc}

\usepackage[english]{babel}

\usepackage{amssymb}
\usepackage{amsthm}
\usepackage{graphicx}
\usepackage{import}
\usepackage{amsmath}
\usepackage{amscd}
\usepackage{mathtools}
\usepackage{verbatim}
\usepackage{tikz-cd}
\usepackage{color}

\newtheorem{theorem}{Theorem}[section]

\newtheorem{thmx}{Theorem}

\newtheorem{proposition}[theorem]{Proposition}
\newtheorem{lemma}[theorem]{Lemma}
\newtheorem{corollary}[theorem]{Corollary}

\theoremstyle{remark}
\newtheorem*{remark}{Remark}

\theoremstyle{remark}
\newtheorem*{notation}{Notation}

\renewcommand{\P}{\mathbb{P}}
\renewcommand{\O}{\mathcal{O}_C}
\newcommand{\OO}{\mathcal{O}}
\newcommand{\SU}{\mathcal{SU}_C(2)}
\newcommand{\SUr}{\mathcal{SU}_C(r)}
\newcommand{\SUw}{\mathcal{SU}_{C_w}(2)^{inv}}
\newcommand{\genSU}{\widehat{\mathcal{SU}}_C(2)}
\newcommand{\PB}{\P_B^{3g-6}}
\newcommand{\PD}{\P_D^{3g-2}}
\newcommand{\PN}{\P_N^{2g-2}}
\newcommand{\PcuatroN}{\P_N^{4}}
\newcommand{\p}{\varphi_D}

\newcommand{\pD}{\varphi_D}
\newcommand{\pPN}{\varphi_{D,N}}
\newcommand{\pPcuatroN}{\varphi_D|_{\PcuatroN}}
\newcommand{\pPB}{\varphi_D(\PB)}
\newcommand{\Kum}{\operatorname{Kum}}

\newcommand{\Ext}{\operatorname{Ext}}
\newcommand{\blowup}{\operatorname{Bl}}
\newcommand{\PGL}{\operatorname{PGL}}
\newcommand{\Cr}{\operatorname{Cr}}

\newcommand{\M}{\mathcal{M}}

\renewcommand{\L}{\mathcal{L}}
\newcommand{\I}{\mathcal{I}}
\newcommand{\Sec}{\operatorname{Sec}}
\newcommand{\pproj}{p_{\P_c}}
\newcommand{\Zeroes}{\operatorname{Zeroes}}
\newcommand{\Pic}{\operatorname{Pic}}
\newcommand{\Jac}{\operatorname{Jac}}

\newcommand{\MGIT}{\M_{0,2g}^{\operatorname{GIT}}}
\newcommand{\MGITn}{\M_{0,n}^{\operatorname{GIT}}}
\newcommand{\MGITsix}{\M_{0,6}^{\operatorname{GIT}}}
\newcommand{\MKnudn}{\overline{\M}_{0,n}}

\newcommand{\Hom}{\text{Hom}}

\renewcommand{\P}{\mathbb{P}}
\renewcommand{\O}{\mathcal{O}}

\newcommand{\Rlin}{\mathcal{R}}

\newcommand{\tto}{\dashrightarrow}

\makeatletter
\renewcommand{\paragraph}{%
      \@startsection{paragraph}{4}%
      {\z@}{1ex \@plus 1ex \@minus .2ex}{-1em}%
      {\normalfont\normalsize\bfseries}%
}
\makeatother

\title{The hyperelliptic theta map and osculating projections.}

\author[M. Bolognesi]{Michele Bolognesi}
\address{Institut Montpellierain Alexander Grothendieck \\ %
Universit\'e de Montpellier \\ %
CNRS \\ %
Case Courrier 051 - Place Eug\`ene Bataillon \\ %
34095 Montpellier Cedex 5 \\ %
France}
\email{michele.bolognesi@umontpellier.fr}

\author{Néstor Fernández Vargas}
\address{Universit\'e de Rennes I, CNRS, IRMAR - UMR 6625, F-35000 Rennes, France}
\email{nestor.fernandez-vargas@univ-rennes1.fr}


\begin{document}

\begin{abstract}
  Let $C$ be a hyperelliptic curve of genus $g \geq 3$. In this paper we give a new geomtric description of the theta map for moduli spaces of rank 2 semistable vector bundles on $C$ with trivial determinant.
  In orther to do this, we describe a fibration of (a birational model of) the moduli space, whose fibers are GIT quotients $(\mathbb{P}^1)^{2g}//\operatorname{PGL(2)}$. Then, we identify the restriction of the theta map to these GIT quotients with some explicit degree two osculating projection. 
  As a corollary of this construction, we obtain a birational inclusion of a fibration in Kummer $(g-1)$-varieties over $\mathbb{P}^g$ inside the ramification locus of the theta map.
\end{abstract}

\maketitle

{}\section{Introduction} \label{sec:introduction}
  Let $C$ be a complex smooth curve of genus $g \geq 3$ and $\SUr$ the (coarse) moduli space of semistable vector bundles of rank $r$ with trivial determinant on $C$. It is well known that this moduli space is a normal, projective, unirational variety of dimension $(r^2 - 1)(g - 1)$. The study of the projective geometry of moduli spaces of vector bundles in low rank and genus has produced some beautiful descriptions, frequently mingling constructions issued in the context of classical algebraic geometry and the geometry of Jacobians and theta functions (\cite{pauly_coble, ortega_coble, desale_ramanan}).


Let $\L$ be the determinant line bundle on $\SUr$ and $\varphi_\L: \SUr \dashrightarrow |\L|^*$ the map induced by global sections of $\L$. The linear system $|\L|^*$ is isomorphic to the $|r \Theta|$ linear series on the Jacobian variety $\Jac(C)$, by the first declination of \it strange duality \rm \cite{beauville_narasimhan_ramanan}. 
%
%
This way, we obtain a (in general) rational map
  
\begin{align*}
  \theta : \SUr \dashrightarrow |r \Theta|,
\end{align*} 

the clebrated theta map, which is canonically identified to $\varphi_{\L}$ \cite{beauville_narasimhan_ramanan}.

  Let us now fix $r = 2$. In this setting, the map $\theta$ is a finite morphism \cite{raynaud}. When $g = 2$, the map $\theta$ is an isomorphism onto $\P^3$ \cite{narasimhan_ramanan_moduli}.
  For $g \geq 3$, the map $\theta$ is an embedding if $C$ is non-hyperelliptic, and it is a 2:1 map if $C$ is hyperelliptic \cite{desale_ramanan, beauville_rang2, brivio_verra, vangeemen_izadi} (see Section \ref{sec:moduli_vector_bundles_theta} for more details). 
  

  The goal of this paper is to describe the geometry associated to the map $\theta$ in the case $r = 2$ when $C$ hyperelliptic. In the non-hyperelliptic case, the papers \cite{bolognesi_conicbundle} and \cite{alzati_bolognesi} outline a connection between the moduli space $\SU$ and the moduli space $\M_{0,n}$ of rational curves with $n$ ordered marked points. A generalization of \cite{alzati_bolognesi} for higher rank vector bundles has been given in \cite{bolognesi_brivio}.
  In the present work, we develop once again the link with the moduli space of pointed rational curves (more precisely with its GIT compactification $\M_{0,n}^{GIT}$). Thanks to some clever description of the GIT compactification in terms of linear systems on the projective space due to Kumar \cite{kumar}, this also offers a new geometric description of the $\theta$-map if $C$ is hyperelliptic. 

\smallskip  
  
  Let $C$ be a hyperelliptic curve of genus $g \geq 3$ and $D$ an effective divisor of degree $g$ on $C$. Let us consider the isomorphism classes of extensions

\begin{align*}
(e) \quad  0 \to \O(-D) \to E_e \to \O(D) \to 0.
\end{align*}
 
 These are classified by the projective space $$\PD := \P \Ext^1(\O(D), \O(-D)) = |K + 2D|^*,$$ where $K$ is the canonical divisor on $C$. Since the divisor $K + 2D$ is very ample, the linear system $|K + 2D|$ embeds the curve $C$ in $\PD$. Let $\PN$ be the span in $\PD$ of the $2g$ marked points $p_1, \ldots, p_{2g}$ on $C$ defined by the effective divisor $N \in |2D|$ ($\PN$ has also a precise description in terms of extensions, see Section \ref{sec3}).  
  
\begin{proposition} \label{prop:fibration_in_M}
  There exists a fibration $p_D:\SU \dashrightarrow |2D| \cong \P^g$ whose general fiber is birational to $\MGIT$. Moreover, we have:
  
  \begin{enumerate}
    \item For every generic divisor $N \in |2D|$, there exists a $2g$-pointed projective space $\PN$ and a rational dominant map $h_N: \PN \dashrightarrow p_D^{-1}(N)$ classifying extension classes, such that the fibers of $h_N$ are rational normal curves passing by the $2g$ marked points.
    \item The family of rational normal curves defined by $h_N$ is the universal family of rational curves over (an open subset of) the generic fiber $\MGIT$.
  \end{enumerate}
\end{proposition}

Our aim is to describe the map $\theta$ restricted to the generic fibers of the fibration $p_D$. To this end, the following construction is crucial:

  Let $p, i(p)$ be two involution-conjugate points in $C$; and consider the line $l \subset \PD$ secant to $C$ and passing through $p$ and $i(p)$. We show that this line intersects the subspace $\PN$ in a point. Moreover, the locus $\Gamma \subset \PN$ of these intersections as we vary $p\in C$ is a rational normal curve passing by the points $p_1, \ldots, p_{2g}$. It follows from Proposition \ref{prop:fibration_in_M}, the map $h_N$ contracts the curve $\Gamma$ onto a point $P \in p_D^{-1}(N) \cong \M_{0,2g}$.

In \cite{kumar}, Kumar defines the linear system $\Omega$ of $(g-1)$-forms on $\P^{2g - 3}$ vanishing with multiplicity $g-2$ at $2g - 1$ general points. He shows that $\Gamma$ induces a birational map $i_{\Omega}: \P^{2g - 3} \dashrightarrow \MGIT$ onto the GIT compactification of the moduli space $\M_{0,2g}$. The partial linear system $\Lambda \subset \Omega$ of forms vanishing with multiplicity $g - 2$ at an additional general point $e \in \P^{2g - 3}$ induces a projection $\kappa : \MGIT \dashrightarrow | \Lambda |^*$. More precisely, $\kappa$ is a 2-to-1 osculating projection centered on the point $w = i_{\Omega}(e)$. We describe birationally the restrictions of $\theta$ to the fibers $p_D^{-1}(N)$ using Kumar's map: 

\begin{thmx} \label{th:theta_is_kappa}
  The map $\theta$ restricted to the fibers $p_D^{-1}(N)$ is the osculating projection $\kappa$ centered at the point $P = h_N(\Gamma)$, up to composition with a birational map.
\end{thmx}

  Furthermore, the image of $\kappa$ is a connected component of the moduli space $\SUw$ of hyperelliptic invariant semistable vector bundles with trivial determinant on $C_w$, where $C_w$ is the hyperelliptic 2-to-1 cover of $\P^1$ ramifying over the $2g$ points defined by $w$. He also proves that the ramification locus of the map $\kappa$ is the Kummer variety $\Kum(C_w) \subset \SUw$.
  These results, combined with Theorem \ref{th:theta_is_kappa}, allow us to give a quite accurate description of the ramification locus of the map $\theta$:
  
\begin{thmx}
  The ramification locus of the map $\theta$ has an irreducible component birational to a fibration in Kummer varieties of dimension $g-1$ over $|2D| \cong \P^g$.
\end{thmx}


\medskip



In low genus we are able to give a more precise description of the Theta map and its interplay with maps classifying extensions. Let $f_D: \PD \dashrightarrow \SU$ denote the natural map sending an extension class onto its rank 2 vector bundle and let us define $\pD$ as $\theta \circ f_D$. 

\begin{thmx}
  Let $C$ be a hyperelliptic curve of genus 3. Then, for generic $N$, the restriction of $\pD$ to the subspace $\PN$ is exactly the composition $\kappa \circ h_N$. If $g=4$ or $5$, then $\pD$ is defined by a (possibly equal) linear subsystem of the one defining $\kappa \circ h_N$, and set-theoretically the base loci of the two linear systems coincide.
\end{thmx}


  \begin{notation} $\P^n = \P(\mathbb{C}^{n+1})$ will denote the $n$-dimensional complex projective space of dim 1 subspaces.  Throughout this paper, a form $F$ of degree $r$ on $\P^n$ will denote element of the vector space $H^0(\P^n, \mathcal{O}_{\P^n}(r)) = \operatorname{Sym}^r(\mathbb{C}^{n+1})^*$. 
  If we fix a basis $x_0, \ldots, x_n$ of $(\mathbb{C}^{n+1})^*$, $F$ is simply a homogeneous polynomial of degree $r$ on $x_0, \ldots, x_n$. Most of the maps in this paper will be rational maps, hence we will often offend good taste by just dropping the adjective \em rational \rm. We apologize for that. 
\end{notation}

{}\section{Moduli of vector bundles} \label{sec:moduli}
  
  We briefly recall here some results about moduli of vector bundles. For a more detailed reference, see \cite{beauville_vectorbundles}.

\subsection{Moduli of vector bundles and the map $\theta$} \label{sec:moduli_vector_bundles_theta}

  Let $C$ be a smooth genus $g$ algebraic curve with $g \geq 2$. 
  Let us denote by $\Pic^d(C)$ the Picard variety of degree $d$ line bundles on $C$. The Jacobian of C is $\Jac(C) = \Pic^0(C)$. The canonical divisor $\Theta \subset \Pic^{g-1}(C)$ is defined set-theoretically as
$$\Theta := \{ L \in \Pic^{g-1}(C) \ | \ h^0(C,L) \not = 0 \}.$$
  Let $\SU$ be the moduli space of semistable rank 2 vector bundles on $C$ with trivial determinant. This variety parametrizes S-equivalence classes of such vector bundles.

  The Picard group $\Pic(\SU)$ is isomorphic to $\mathbb{Z}$, and it is generated by the \em determinant line bundle \rm $\mathcal{L}$  \cite{drezet_narasimhan}. For every $E \in \SU$, let us define the \emph{theta divisor}
  
$$\theta(E) := \{ L \in \Pic^{g-1}(C) \ | \ h^0(C, E \otimes L) \not = 0 \}. $$

  In the rank 2 case, $\theta(E)$ is a divisor in the linear system $|2 \Theta| \cong \P^{2^g - 1}$ and $|2\Theta|$ is isomorphic to the linear system $|\L|^*$ \cite{beauville_narasimhan_ramanan}. It is well known that we can identify the map $\SU \to |\L|^*$ with the Theta map

\begin{eqnarray}
\theta : \SU & \to & |2 \Theta|;\\
E & \mapsto & \theta(E).
\end{eqnarray}  


In rank 2, the map $\theta$ is a finite morphism. If $C$ is not hyperelliptic, $\theta$ is known to be an embedding \cite{brivio_verra, vangeemen_izadi}. This is also the case in genus 2, where $\theta$ is an isomorphism onto $\P^3$ \cite{narasimhan_ramanan_moduli}. If $C$ is hyperelliptic of genus $g \geq 3$, we have that $\theta$ factors through the involution $$E \mapsto i^*E$$ induced by the hyperelliptic involution $i$, embedding the quotient $\SU/i^*$ into $|2 \Theta|$ \cite{desale_ramanan, beauville_rang2}. An interesting explicit description of the image of the hyperelliptic theta map is given in \cite{desale_ramanan}.

\subsection{The classifying maps} \label{sec:classifying_maps}

Let $D$ be a general degree $g$ effective divisor on $C$. Let us consider isomorphism classes of extensions

\begin{align*}
(e) \quad  0 \to \O(-D) \to E_e \to \O(D) \to 0.
\end{align*}

These extensions are classified by the $(3g - 2)$-dimensional projective space

\begin{align*}
  \PD := \P \Ext^1(\O(D), \O(-D)) = |K + 2D|^*,
\end{align*}
where $K$ is the canonical divisor of $C$.
  The divisor $K + 2D$ is very ample and embeds $C$ as a degree $4g - 2$ curve in $\PD$. 
  Let us define the rational surjective classifying map
  
\begin{align*}
  f_D : \PD  &\tto \SU 
\end{align*}

which sends the extension class $(e)$ to the vector bundle $E_e$. The composed map 

\begin{align*}
  \p : = \theta \circ f_D : \PD \tto | 2 \Theta |
\end{align*}

can be described in terms of polynomial maps. From \cite[Thm. 2]{bertram} we have an isomorphism

\begin{align} 
  H^0(\SU, \mathcal{L}) \cong H^0(\PD, \I_C^{g-1}(g)),
\end{align}

where $\I_C$ is the ideal sheaf of $C$. In particular we have

\begin{theorem}\label{th:bertram}
The map $\p$ is given by the linear system $|\I_C^{g-1}(g)|$ of forms of degree $g$ vanishing with multiplicity at least $g - 1$ on $C$. 
\end{theorem} 

Let us denote by $\Sec^n(C)$ the variety of $(n + 1)$-secant $n$-planes on $C$.
We have that the singular locus of $\Sec^{n+1}(C)$ is the secant variety $\Sec^{n}(C)$ for every $n$.
The linear system $|\I_C^{g-1}(g)|$ is characterized as follows:

\begin{proposition}[\protect{\cite[Lemma 2.5]{alzati_bolognesi}}] \label{prop:linear_systems_coincide}
  The linear system $|\I_C^{g-1}(g)|$ and $|\I_{\Sec^{g-2} (C)}(g)|$ on $\PD$ are the same.
\end{proposition}
\begin{proof}
  We reproduce here the proof for the reader's convenience.
  The elements of both linear systems can be seen as symmetric $g$-linear forms on the vector space $H^0(C, K + 2D)^*$. 
  Let $F$, $G$ be such forms. 
  Then, $F$ belongs to $|\I_C^{g-1}(g)|$ (resp. $G$ belongs to $|\I_{\Sec^{g-2} (C)}(g)|$) if and only if
\begin{align*}
  F(p_1, \ldots, p_g) &= 0 \quad  \text{ for all } p_k \in C \text{ such that } p_i = p_j \text{ for some } 1 \leq i,j \leq g \\
  G(p, \ldots, p) &= 0 \quad \text{ for any linear combination } p = \sum_{k=1}^{g-1} \lambda_i p_i \text{, where } p_i \in C .
\end{align*}
  One can show that these conditions are equivalent by exhibiting appropriate choices of $\lambda_i$. 
\end{proof}

\subsection{The exceptional fibers of the classifying map $f_D$}
  Since $\dim \SU = 3g - 3$, the generic fiber of $f_D$ has dimension one. The set of stable bundles for which $\dim(f_D^{-1}(E)) > 1$ is a proper subset of $\SU$. For simplicity, let us define the "Serre dual" divisor
  
\begin{align*}
  B := K - D
\end{align*} 

  with $\deg(B) = g-2$. As in the previous paragraphs, the isomorphism classes of extensions

\begin{align*}
  \quad  0 \to \O(-B) \to E \to \O(B) \to 0
\end{align*}

 are classified by the projective space 
 
 $$\PB := \P \Ext^1 (\O(B), \O(-B)) = |K + 2B|^*,$$ 
 
which is endowed with the rational classifying map $f_B :  \PB  \tto \SU$ defined in the same way as $f_D$.
   
\begin{proposition} \label{prop:exceptional_fibers}
  Let $E \in \SU$ be a stable bundle. Then 
\begin{align*}
  \dim(f_D^{-1}(E)) \geq 2 \qquad  \text{if and only if} \qquad   E \in \overline{f_B(\PB)}.
\end{align*}
\end{proposition}
\begin{proof}
  Let $E$ be a stable bundle. Then, by Riemann-Roch and Serre duality theorems, the dimension of $f^{-1}_D(E)$ is given by
\begin{align*}
   h^0(C, E \otimes \O(D)) &= h^0(C, E \otimes \O(B)) + 2g - 2(g-1) 
\\ &= h^0(C, E \otimes \O(B)) + 1
\end{align*}
Thus, $\dim(f_D^{-1}(E)) > 2$ if and only if there exists a non-zero sheaf morphism $\O(-B) \to E$. This is equivalent to $E \in \overline{f_B(\PB)}$.
\end{proof}

  If $g > 2$, the divisor $|K + 2B|$ embeds $C$ as a degree $4g - 6$ curve in $\PB$ (recall that $\P \Ext^1 (\O(B),\O(-B)) = |K + 2 B|^*$). 
  Again by Theorem \ref{th:bertram}, the map $\varphi_B$ is given by the linear system $|\I_C^{g-3}(g-2)|$.
  Moreover, by \cite[Theorem~4.1]{pareschi_popa} this linear system has projective dimension $\left( \sum_{i = 0}^{g-2} {\binom{g}{i}} \right) - 1$ .

  Let us denote by $\P_c$ the linear span of $\theta(\overline{f_B(\PB)})$ in $|2 \Theta|$.  
  Since the map $\theta$ is finite, $\P_c$ has projective dimension $\left[ \sum_{i = 0}^{g-2} {\binom{g}{i}} \right] - 1$, and Proposition \ref{prop:exceptional_fibers} also applies to $\p$: the fibers of $\p$ with dimension $\geq 2$ are those over $\P_c$.

{}\section{A linear projection in $|2 \Theta|$ }\label{sec3}

The goal of this Section is to describe the map $\SU \to \P^g$ whose fibers will be birational - and in some cases even biregular - to the GIT compactification of the moduli space of $2g$-pointed rational curves. In order to do this, we describe the projection with center $\P_c$, seen as a linear subspace of $|2 \Theta|$. The fibers of this map are what we are looking for, and so far we are not resricting to the case where $C$ is hyperelliptic.

\medskip

Let $\pproj$ be the linear projection in $|2 \Theta|$ with center $\P_c$. Recall that $\dim \P_c = \left[ \sum_{i = 0}^{g-2} \binom{g}{i} \right] - 1$. A straightforward calculation shows that the supplementary linear subspaces of $\P_c$ in $|2 \Theta|$ are of projective dimension $g$. Thus, the image of $\pproj$ is a $g$-dimensional projective space.
Let us set

$$\genSU := \SU \setminus (\Kum(C) \cup \overline{\pPB}).$$ 

This is the open subset of $\SU$ we will be mostly concerned by.
Recall that the space $H^0(C, E \otimes \O(D))$ has dimension 2 for $E \in \genSU$. Consequently, we can pick two sections $s_1$ and $s_2$ that constitute a basis for this space.

\begin{theorem} \label{th:proyeccion}
  The image of the projection $\pproj$ can be identified with the linear system $|2 D|$ on $C$, in a way such that the restriction of the projection $\pproj$  to $\theta(\genSU)$ coincides with the map

\begin{align*}
  \theta(\genSU) \to |2D|  \\
  \theta(E) \mapsto \Zeroes(s_1 \wedge s_2)
\end{align*}
\end{theorem}

\begin{proof}
  This result was proved in \cite{alzati_bolognesi} for $C$ non hyperelliptic, but the proof extends harmlessly to the hyperelliptic case. We will mention explicitly where the proof for $C$ hyperelliptic differs.
  The Picard variety $\Pic^{g-1}(C)$ contains a model $\widetilde{C}$ of $C$, made up by line bundles of type $\O(B + p)$, with $p \in C$. The span of $\widetilde{C}$ inside $|2 \Theta|^*$ corresponds to the complete linear system $|2 D|^*$. 
  Moreover, the linear span of $\widetilde{C}$ is the annihilator of $\P_c$. In particular, the projection $p_{\P_c|\theta(\genSU)}$ determines a hyperplane in the annihilator of $\P_c$, which is a point in $|2 D|$. This projection can be identified with the map

\begin{align*}
  \pproj|_{\theta(\genSU)} : \theta(\genSU) & \to |2D|,  \\
  \theta(E) & \mapsto \Delta(E) ,
\end{align*}

where $\Delta(E)$ is the divisor defined by

\begin{align}
  \Delta (E):= \{p \in C \ | \ h^0(C,E \otimes \O(B + p)) \not = 0 \}.
\end{align}

Equivalently, we have that $\Delta(E) = \theta(E) \cap \widetilde{C}$. Now, in order to adapt to the hyperelliptic case, it is enough to observe that since $\theta(E) = \theta(i^*E)$, we directly obtain that $\Delta(E) = \Delta(i^* E)$.
Finally, an easy Riemann-Roch argument shows that that $\Delta(E)$ is the divisor of zeroes of $s_1 \wedge s_2$.
\end{proof}

  Recall that the linear system $|K + 2D|$ embeds the curve $C$ in the projective space $\PD$. Let $N \in |2D|$ be a generic effective reduced divisor and consider the linear span $\langle N \rangle \subset \PD$. The annihilator of $\langle N \rangle$ is the vector space $H^0(C, 2D + K - N)$, which has dimension $g$. In particular, the linear span $\langle N \rangle$ has dimension $(3g - 2) - g = 2g - 2$.
  Let us write $$\PN := \langle N \rangle \subset \PD.$$ 
  We will study the classifying map $\p$ in relation with a fibration $\SU \to \P^g$ by considering the  restrictions of $\p$ to $\PN$, as $N$ varies in the linear system $|2D|$. The spaces $\PN$ have a very explicit description in terms of extension classes (see \cite{lange_narasimhan}).

\begin{notation}
  For simplicity, let us write $\pPN$ for the restricted map $\p|_{\PN}$.
\end{notation}

\begin{proposition} \label{prop:fibre=image}
  Let $N$ in $|2D|$ be a general divisor on $C \subset \PD$. Then, the image of
\begin{align*}
  \pPN : \PN \tto \theta(\SU)
\end{align*}
is the closure in $\theta(\SU)$ of the fiber over $N \in |2D|$ of the projection $\pproj$.
\end{proposition}

\begin{proof}
  Let $(e) \in \PD$ be an extension 
\begin{align*}
(e) \quad  0 \rightarrow \O(-D) \xrightarrow{i_e} E_e \xrightarrow{\pi_e} \O(D) \to 0.
\end{align*}
  By \cite[Proposition 1.1]{lange_narasimhan}, we have that $e \in \PN$ if and only if there exists a section 
$$\alpha \in H^0(C, \Hom(\O(-D),E))$$
such that $\Zeroes(\pi_e \circ \alpha) = N$. 
  This means that $\alpha$ and $i_e$ are two independent sections of $E_e \otimes \O(D)$ with $\Zeroes(\alpha \wedge i_e) = N$. Consequently, $\theta(E_e) = \pPN (e)$ is projected by $\pproj$ on $N \in |2D|$ by Theorem \ref{th:proyeccion}. Hence, the image of $\pPN$ is contained in $\overline{\pproj^{-1}(N)}$.

  Conversely, by the proof of Theorem \ref{th:proyeccion}, we have that for every bundle $E \in \genSU$, $\theta(E)$ is projected by $\pproj$ to a divisor $\Delta(E) \in |2D|$.
  The argument used above implies that the fiber $\p^{-1}(\theta(E)) = f_D^{-1}(E)$ of such bundle is contained in $\P^{2g - 2}_{\Delta(E)}$. Consequently, the fiber of a general divisor $N \in |2D|$ by $\pproj$ is contained in the image of $\pPN$.
\end{proof}
  
{}\section{The modular fibration $\SU \to \P^g$} \label{sec:restriction_map}
  
Let $C$ be a smooth genus $g \geq 3$ curve (not necessarily hyperelliptic). Let $D$ be a general degree $g$ effective divisor on $C$. Let $N = p_1 + \cdots + p_{2g}$ be a general divisor in the linear system $|2D|$. Consider the span $\PN$ in $\PD$ of the $2g$ marked points $p_1, \ldots,  p_{2g}$. In this Section, building on \cite{bolognesi_BLMS} and \cite{alzati_bolognesi}, we will give more information about restricted map $$\pPN = \p|_{\PN}: \PN \tto \SU.$$ In particular, we will explain the interplay between these maps, rational normal curves  in $\PN$ and moduli spaces of rational pointed curves.

\subsection{Linear systems in $\PN$ contracting rational normal curves.}  
  
 
 Recall that the secant variety $\Sec^{g - 2}(C)$ is the base locus for $\p$ (see Theorem \ref{th:bertram} and Proposition \ref{prop:linear_systems_coincide}). Hence we will distinguish the following two secant varieties in $\PN$:
 
\begin{align*}
  \Sec^N &:= \Sec^{g-2}(C) \cap \PN \ , \\
  \Sec^{g-2}(N) &:= \bigcup_{\substack{M \subset N \\ \#M = g - 1}} \operatorname{span} \{M \} \ .
\end{align*}

  Note that, since the points of $N$ are already in $\PN$, we have the inclusion $\Sec^{g-2}(N) \subset \Sec^N$. 
We will also need to consider the linear systems on $\PN$

\begin{align*}
  \I_{\Sec^N}(g) \ , \quad \mathrm{and}\ \quad
  \I_{\Sec^{g-2}(N)}(g) 
\end{align*}

  of forms of degree $g$ vanishing on the corresponding secant varieties. The previous inclusion of secant varieties implies that $\I_{\Sec^N}(g)$ is in general a linear subsystem of $\I_{\Sec^{g-2}(N)}(g)$.

\begin{lemma} \label{lem:base_locus}
  The restricted map $\pPN$ is given by a linear subsystem $\Rlin$ of $|\I_{\Sec^N}(g)|$. 
\end{lemma}

\begin{proof}
  This is a direct consequence of Theorem \ref{th:bertram} and Proposition \ref{prop:linear_systems_coincide}.
\end{proof}

\subsection{Moduli spaces of pointed rational curves} \label{sec:moduli_pointed}
In this Section we will outline the relation between the restricted map $\pPN$ and the moduli spaces of pointed rational curves. 

\subsubsection{Two compactifications of $\M_{0,n}$}
  The moduli space $\M_{0,n}$ of ordered configurations of $n$ disctinct points on the projective line is not compact.
  We will consider two compactifications of $\M_{0,n}$. The first one is the GIT quotient $$\MGITn := (\P^1)^n // \PGL(2, \mathbb{C})$$ of $(\P^1)^n$ by the diagonal action of $G=\PGL(2, \mathbb{C})$ for the natural $G$-linearization of the line bundle $L = \boxtimes_{i=1}^n\O_{\P^1}(1)$ (see \cite{dolgachev_ortland}). 
The quotient $\MGITn$ is naturally embedded in the projective space $\P ( H^0 ((\P^1)^n, L )^{G} )$ of invariant sections.
  
The second one is the Mumford-Knudsen compactification $\MKnudn$ \cite{knudsen_projectivity}. The points of $\MKnudn$ represent isomorphism classes of stable curves.
  More details on these constructions can be found in \cite{kapranov_veronese} and \cite{knudsen_projectivity}.

Both $\MGIT$ and $\MKnudn$ contain $\M_{0,n}$ as an open set. However, the Mumford-Knudsen compactification is finer on the boundary: there exists a contraction dominant morphism $$c_n: \MKnudn \to \MGITn$$ contracting some components of the boundary of $\MKnudn$, that restricts to the identity on $\M_{0,n}$ \cite{bolognesi_BLMS}. Moreover, we will denote by 

$$\lambda_k: \overline{\M}_{0,n} \to \overline{\M}_{0,n-1},$$

for $k = 1,\ldots,n$, the forgetful morphism that forgets the labelling of the $k$-th point.

\subsubsection{A variety of rational normal curves}

  Let $e_1, \ldots e_{n} \in \P^{n - 2}$ be $n$ points in general position.
  Let $\mathcal{H}$ be the Hilbert scheme of subschemes of $\P^{n-2}$.
  Let $V_0(e_1, \ldots, e_n) \subset \mathcal{H}$ be the subvariety of rational normal curves in $\P^{n - 2}$ passing through the points $e_1, \ldots, e_n$, and let $V(e_1, \ldots, e_n)$ be the closure  of $V_0(e_1, \ldots, e_n)$ inside the Hilbert scheme of subschemes of $\P^{n-2}$.
  The boundary $V(e_1, \ldots, e_n) \setminus V_0(e_1, \ldots, e_n)$ consists on reducible rational normal curves, i.e. reducible non-degenerate curves of degree $n$ such that each component is a rational normal curve in its projective span.

  There exists an isomorphism $V_0(e_1, \ldots, e_{n}) \cong \M_{0,n}$ (see \cite{dolgachev_ortland}) associating to a rational normal curve passing by $e_1, \ldots, e_n$ the corresponding ordered configuration of $n$ points in $\P^1$. This can be extended \cite{kapranov_veronese} to an isomorphism between $V(e_1, \ldots, e_n)$ and $\overline{\M}_{0,n}$.

\subsubsection{The blow-up construction}
  The following construction is due to Kapranov \cite{kapranov_chow}: let $e_1, \ldots, e_{n-1} \in \P^{n-3}$ be $(n-1)$ points in general position. Consider the following sequence of blow-ups:
  
\begin{enumerate}
  \item Blow-up the points  $e_1, \ldots, e_{n-1}$.
  \item Blow-up the proper transforms of lines spanned by pairs of points from $\{e_1, \ldots, e_{n-1} \}$.
  \item Blow-up the proper transforms of planes spanned by triples of points from $\{e_1, \ldots, e_{n-1} \}$.
  \item[$\vdots \ $]
  \item[$(n-4)$.] Blow-up the proper transforms of linear subspaces spanned by $(n-4)$-ples of points from $\{e_1, \ldots, e_{n-1} \}$.
\end{enumerate}

  Let $\blowup(\P^{n-3})$ be the $(n - 3)$-variety obtained in this way, and $b : \blowup(\P^{n-3}) \to \P^{n-3}$ the composition of this sequence of blow-ups. We will call this map the \emph{Kapranov blow-up map centered in the points $e_1, \ldots, e_{n - 1}$}.

\begin{theorem}[Kapranov \cite{kapranov_chow}]
  Let $n \geq 4$. Then, the moduli space $\MKnudn$ is isomorphic to $\blowup(\P^{n-3})$.
\end{theorem}

  Moreover, the images by $b$ of the fibres of the map $\lambda_k$ over the points in the open set $\M_{0,n-1} \subset \overline{\M}_{0,n-1}$ are the rational normal curves in $\P^{n-3}$ passing through the $n - 1$ points $e_1,\ldots, e_{n-1}$ (see \cite[Prop. 3.1]{keel_mckernan}).

\subsubsection{The Cremona inversion}
  Let $e_1, \ldots e_{n - 1} \in \P^{n - 3}$ in general position. Without loss of generality, we may assume $e_k = [0: \cdots : 1 : \cdots : 0]$ for $k = 1, \ldots, n - 2$; and $e_{n - 1} = [1: \cdots : 1]$. The \emph{Cremona inversion with respect to $e_{n-1}$} is the birational map
  
\begin{align*}
  \Cr_{n - 1} : \mathbb{P}^{n - 3} & \tto  \mathbb{P}^{n - 3} \\
  \left[x_0 : \dots : x_{n - 3}\right] & \mapsto  \left[1/x_0 : \dots : 1/x_{n - 3} \right].
\end{align*}

On $\P^2$ the Cremona inversion is given by the linear system of quadrics passing through $e_1,\ e_2$ and $e_3$, on $\P^3$ by the cubics that vanish on the six secant lines
of $e_1,\ e_2,\ e_3$ and $e_4$, and so on.
  The Cremona inversion has the following property: any non-degenerate rational normal curve passing through the points $e_1, \ldots, e_{n - 1}$ is transformed into a line passing by the point $\Cr_{n - 1}(e_{n - 1})$.  
  Let $\tau_{n - 1} : \P^{n - 3} \tto \P^{n - 4}$ be the linear projection with center $\Cr_{n - 1}(e_{n - 1})$.
  From the previous property, we obtain that the composition $\tau_{n - 1} \circ \Cr_{n - 1}$ contracts non-degenerate rational normal curves passing through $e_1, \ldots, e_{n - 1}$.

  Let $k \in \{1,\ldots,n-1\}$. It is straightforward to see that one can let $e_k$ play the role of $e_{n - 1}$ in the definition of $\Cr_{n - 1}$, and define similarly the Cremona inversion $\Cr_k$. 
  Let $\tau_k : \P^{n - 3} \tto \P^{n - 4}$ be the linear projection with center $\Cr_k(e_{k})$.

\begin{lemma} \label{lem:composition_contracts}
  Let $e_1, \ldots e_{n - 1} \in \P^{n - 3}$ in general position. 
  Then, the composition $\tau_k \circ \Cr_k$ contracts the non-degenerate rational normal curves passing through $e_1, \ldots, e_n$.
\end{lemma}

Let us denote $H_t$, for $t\neq k$, the hyperplane in $\P^{n-3}$ spanned the points $e_i$, with $i\neq k,t.$ There are $n-2$ such hyperplanes and each one is contracted to a point by $\Cr_k$.

\begin{proposition}[ \cite{kapranov_veronese}] \label{prop:cuadrado_kapranov}
  Let $e_1, \ldots e_{n - 1} \in \P^{n - 3}$ in general position. 
  Let $k \in \{1,\ldots,n - 1\}$. Then, the following diagram is commutative:

\begin{equation*} 
  \begin{tikzcd}[column sep=huge]
    \MKnudn \arrow[r,"\lambda_k"] \arrow[d,"b",swap] & \overline{\mathcal{M}}_{0,n-1}  \arrow[d,"b_k"] \\
 \P^{n - 3}  \arrow[r,dashed,"\tau_k \ \circ \ \Cr_{k}"]  &   \P^{n - 4}
  \end{tikzcd}
\end{equation*}

  Here, the map $b_k$ is the Kapranov blow-up map centered in the images of the hyperplanes $H_t$, for $1\leq t \leq n-1$ and $t \neq k$, by $\tau_k \circ \Cr_k$.
\end{proposition}

\begin{remark}
We observe here a little subtlety. We only get here $n-1$ forgetful maps through Cremona transformations, because we are tacitly assuming that Kapranov's blow-up construction of $\MKnudn$ labels with integers from 1 to $n-1$ the points $e_1,\dots,e_{n-1}$ of the projective base of $\P^{n-3}$, and labels as $n$ the last point (which is free to move inside the $\P^{n-3}$ birational model of $\MKnudn$). This is due to this small asymmetric aspect of Kapranov's construction, but it is clear that one could assume that the last, free, point is labeled with any $k\in\{ 1, \dots, n-1\}$, and obtain other forgetful maps.
\end{remark}

\subsubsection{Rationalizations of $\M^{\operatorname{GIT}}_{0,2g}$}

  Let $g \geq 3$, and let $e_1, \ldots, e_{2g - 1} \in \P^{2g - 3}$ in general position. 
  Let $\Omega$ be the linear system of $(g-1)$-forms in $\P^{2g-3}$ vanishing with multiplicity $g-2$ in $e_1, \ldots, e_{2g - 1} \in \P^{2g - 3}$. 
 
\begin{theorem}[\cite{kumar}] \label{th:kumar_Omega}
  Let $g \geq 3$, and let $e_1, \ldots, e_{2g - 1} \in \P^{2g - 3}$ be in general position. 
  Then, the rational map $$i_{\Omega}:\P^{2g-3} \tto  \Omega^*$$ induced by the linear system $\Omega$ maps $\P^{2g-3}$ birationally onto $\M^{\operatorname{GIT}}_{0,2g}$.
\end{theorem}

 We also observe that the contraction map $c_{2n}$ can also be described in terms of Kumar's linear system $\Omega$:

\begin{lemma}[\cite{bolognesi_BLMS}]\label{contrLS}
  Let $g \geq 3$, and let $e_1, \ldots, e_{2g - 1} \in \P^{2g - 3}$ in general position. 
  Then, the following diagram is commutative:

\begin{equation*} 
  \begin{tikzcd}[column sep=huge]
    \overline{\mathcal{M}}_{0, 2g} \arrow[rd,"c_{2n}"] \arrow[d,"b",swap]  \\
 \P^{2g-3}  \arrow[r,dashed,"i_{\Omega}"]  & \M^{\operatorname{GIT}}_{0,2g}
  \end{tikzcd}
\end{equation*}

  Here, the map $b$ is the Kapranov blow-up map centered in $p_1, \ldots, p_{2g - 1}$.
\end{lemma}


  Let $e_0 \in \P^{2g - 3}$ such that $w = i_{\Omega}(e_0)$ lies in the open set $\M_{0,2g} \subset \M^{\operatorname{GIT}}_{0,2g}$. 
  The point $w$ can represent a hyperelliptic genus $(g - 1)$ curve $C_w$ (namely the double cover of $\P^1$ ramifying in the $2g$ marked points) together with an ordering of the Weierstrass points.
  Let $\SUw$ be the moduli space of rank 2 semistable vector bundles with trivial determinant over the curve $C_w$, that are invariant w.r.t. the hyperelliptic involution.

  Consider the partial linear system $\Lambda$ of $\Omega$ consisting of the $(g - 1)$-forms in $\P^{2g-3}$ vanishing with multiplicity $g - 2$  in all the points $e_0, e_1, \ldots, e_{2g - 1}$. Let $\kappa: \MGIT \tto  \Lambda ^*$ be the rational projection induced by the linear system $\Lambda$.

\begin{theorem}[\cite{kumar}]\label{th:kumar_double_cover}
  Let $g \geq 3$, and let $e_1, \ldots, e_{2g - 1} \in \P^{2g - 3}$ be $2g - 1$ points in general position.  
  Let $e_0 \in \P^{2g - 3}$ such that $w = i_{\Omega}(e_0)$ lies in the open set $\M_{0,2g} \subset \M^{\operatorname{GIT}}_{0,2g}$.
  Then, the map $\kappa$ induced by the linear system $\Lambda$ is of degree 2 onto a connected component of the moduli space $\SUw$.
  Furthermore, the map $\kappa$ ramifies along the Kummer variety $\Kum(C_w) \subset \SUw$.

\begin{equation*}
  \begin{tikzcd}
    \P^{2g-3} \arrow[r,dashed,"i_\Omega"] \arrow[rd,swap,dashed,"i_\Lambda"] & \MGIT\arrow[d,dashed,"\kappa"] \\
    &                                          \SUw 
  \end{tikzcd}
\end{equation*}
\end{theorem}

\subsection{Forgetful linear systems and $\M^{\operatorname{GIT}}_{0,2g}$} \label{sec:h_N}
  Let $C$ be a smooth genus $g \geq 3$ curve (not necessarily hyperelliptic).
  Let $D$ be a general degree $g$ effective divisor on $C$. 
  Let $N = p_1 + \cdots + p_{2g}$ be a general reduced divisor in the linear system $|2D|$. Consider the span $\PN$ in $\PD$ of the $2g$ marked points $p_1, \ldots,  p_{2g}$. 

\medskip

  We will now apply the discussion of Section \ref{sec:moduli_pointed} to the general points $p_1, \ldots, p_{2g}$ in the projective space $\PN$, taking $n = 2g + 1$. For every $k = 1, \ldots, 2g$, we can compose Prop. \ref{prop:cuadrado_kapranov} and Lemma \ref{contrLS} and get a commutative diagram
  
\begin{equation}\label{gitcontraction} 
  \begin{tikzcd}[column sep=huge]
    \overline{\mathcal{M}}_{0, 2g + 1} \arrow[r,"\lambda_k"] \arrow[d,"b",swap] & \overline{\mathcal{M}}_{0,2g}  \arrow[d,"b_k"] \arrow[rd,"c_{2n}"] \\
 \PN  \arrow[r,dashed,"\tau_k \ \circ \ \Cr_{k}"]  &   \P^{2g - 3} \arrow[r,dashed,"i_{\Omega}"]  & \MGIT
  \end{tikzcd}
\end{equation}

  where $\Omega$ is the linear system of $(g-1)$-forms in $\P^{2g-3}$ vanishing with multiplicity $g-2$ at the $2g - 1$ points $\tau_{k} \circ \Cr_{k}(H_i)$, with $i\neq k$ and $1 \leq i \leq 2g$.  
Let us define the rational map 

$$h_N : \PN \tto |\I_{\Sec^{g-2}(N)}(g)|*.$$ 


\begin{proposition}[\cite{bolognesi_BLMS}] \label{prop:h_N_is_composition}
  Let $N = p_1 + \cdots + p_{2g}$ be a general reduced divisor in the linear system $|2D|$. Then, the map $h_N$ coincides with the composition $i_{\Omega} \circ \tau_k \ \circ \ \Cr_{k}$ for every $k = 1, \ldots, 2g$.
  In particular, the composition $i_{\Omega} \circ \tau_k \ \circ \ \Cr_{k}$ does not depend on $k$.
\end{proposition}

This is due to the fact that the linear system $|\I_{\Sec^{g-2}(N)}(g)|$ is invariant w.r.t the action of the symmetric group $\Sigma_{2g}$ that operates on $\PN$ by linear automorphisms. Let us put together the results of Lemma \ref{lem:composition_contracts}, Theorem \ref{th:kumar_Omega} and Proposition \ref{prop:h_N_is_composition} in the following Proposition:

\begin{proposition} \label{prop:RNC}
  The image of $h_N$ is isomorphic to the GIT moduli space $\MGIT$ of ordered configurations of $2g$ points in $\P^1$. The map $h_N$ is dominant and its general fiber is of dimension 1. More precisely, $h_N$ contracts every rational normal curve $Z$ passing through the $2g$ points $N$ to a point $z$ in $\MGIT$. This point represents an ordered configuration of the $2g$ points $N$ on the rational curve $Z$.
\end{proposition}

This is why these maps where dubbed "forgetful linear systems". In fact the rational normal curves passing through the $2g$ points make up the universal curve over an open subset of $\MGIT$.

\medskip

Since $\Rlin$ is a linear subsystem of $|\I_{\Sec^{g-2}(N)}(g)|$ by Lemma \ref{lem:base_locus}, we have that $\pPN$ factors through $h_N$: 

\begin{equation} \label{diag:factorizacionporhn}
  \begin{tikzcd}
    \PN \arrow[r,dashed,"h_N"] \arrow[rd,swap,dashed,"\pPN"] & \MGIT  \arrow[d,dashed] \\
    &                                           \theta(\SU) 
  \end{tikzcd}
\end{equation}


\subsubsection{A comparison of base loci.}

For future use, we need to compare the locus $\Sec^{g-2}(N)$ and the more intricate locus $\Sec^N$ obtained by intersecting the base locus of $\p$ with $\PN$. This section is devoted to this comparison.

\medskip

  By definition, the points in $\Sec^N$ are given by the intersections $\langle L_{g-1} \rangle \cap \PN $, where $L_{g-1}$ is an effective divisor of degree $g-1$ and $\langle L_{g-1} \rangle$ is its linear span in $\PD$.
  If $L_{g-1}$ is contained in $N$, it is clear that $\langle L_{g-1} \rangle \subset \Sec^{g-2}(N) \subset \PN$.

\begin{lemma} \label{lemma:nonempty_intersection}
  Let $L_{g-1}$ be an effective divisor on $C$ of degree $g - 1$, not contained in $N$. Then, 
\begin{align*}
  \langle L_{g-1} \rangle \cap \PN \not = \phi \qquad  \text{if and only if} \qquad \dim|L_{g-1}| \geq 1.
\end{align*}
  Moreover, if the intersection is non-empty, we have that $$\dim (\langle L_{g-1} \rangle \cap \PN) = \dim |L_{g-1}| - 1.$$
\end{lemma}

\begin{proof}
  First, let us suppose that $L_{g-1}$ and $N$ have no points in common. 
  The vector space $V := H^0(C, 2D + K - L_{g-1})$ is the annihilator of the span $\langle L_{g-1} \rangle$ in $\PD$. By the Riemann-Roch theorem, we see that $V$ has dimension $2g$, hence $$\dim \langle L_{g-1} \rangle = (3g - 2) - 2g = g - 2.$$
  Let $d$ be the dimension of the span $\langle L_{g-1}, N \rangle$ of the points of $L_{g-1}$ and $N$. Since the dimension of $\PN = \langle N \rangle$ is $2g-2$, we have that $d \leq (g-2) + (2g - 2) + 1 = 3g-3$, where the equality holds iff $\langle L_{g-1} \rangle \cap \PN$ is empty.

  In particular, this intersection is non-empty iff $d \leq 3g -4$. Since $\dim |K + 2D|^* = \dim \PD = 3g - 2$, this is equivalent to the annihilator space $$W :=  H^0(C, 2D + K - L_{g-1} - N) = H^0(C, K - L_{g-1})$$ being of dimension $\geq 2$. By Riemann-Roch and Serre duality, we obtain that this condition is equivalent to $\dim |L_{g-1}| \geq 1$. 

  More precisely, let us suppose that $\langle L_{g-1} \rangle \cap \PN$ is non-empty and let $e := \dim (\langle L_{g-1} \rangle \cap \PN)$. Then, we have that $$d = 3g - 3 - (e + 1),$$ and the annihilator space $W$ is of dimension $2 + e$. Again by a Riemann-Roch computation, we conclude that $e = \dim |L_{g-1}| - 1$.

  Finally, if $L_{g-1}$ and $N$ have some points in common, we have to count them only once when defining the vector space $W$ to avoid requiring higher vanishing multiplicity to the sections.
 
\end{proof}

  From this Lemma, we conclude that $\Sec^{g-2}(N)$ is a proper subset of $\Sec^N$ if and only if there exists a divisor $L_{g-1}$ not contained in $N$ with $\dim|L_{g-1}| \geq 1$. By the Existence Theorem of Brill-Noether theory (see \cite[Theorem 1.1, page 206]{arbarello_cornalba}) this is possible only if $g \geq 4$ in the non-hyperelliptic case, whereas such a linear system may exist also when $g=3$ when $C$ is hyperelliptic. We will discuss the first low genera cases in Section \ref{sec:lowgenera}.

\section{The hyperelliptic case} \label{sec:gamma}

From now on, $C$ will be a \emph{hyperelliptic} curve of genus $g \geq 3$. 

\smallskip

As we have seen in Lemma \ref{lem:base_locus}, the base locus of the map $\pPN$ contains $\Sec^N$.
  We have seen that the secant variety $\Sec^{g-2}(N)$ is contained in $\Sec^N$ and that this inclusion is strict for $g \geq 4$ in the non-hyperelliptic case.
  
\subsection{A rational normal curve coming from involution invariant secant lines.}  

  In the hyperelliptic case, we have an additional base locus for every $g \geq 3$, which appears due to the hyperelliptic nature of the curve. This locus arises as follows: for each pair $\{ p, i(p) \}$ of involution-conjugate points in $C$, consider the secant line $l$ in $\PD$ passing through the points $p$ and $i(p)$. Let $Q_p$ be the intersection of the line $l$ with $\PN$.
  Let us define $\Gamma \subset \PN$ as the locus of intersection points $Q_p$ when $p$ moves inside $C$.

\begin{figure}
      \centering
          \def\svgwidth{0.7\columnwidth}
          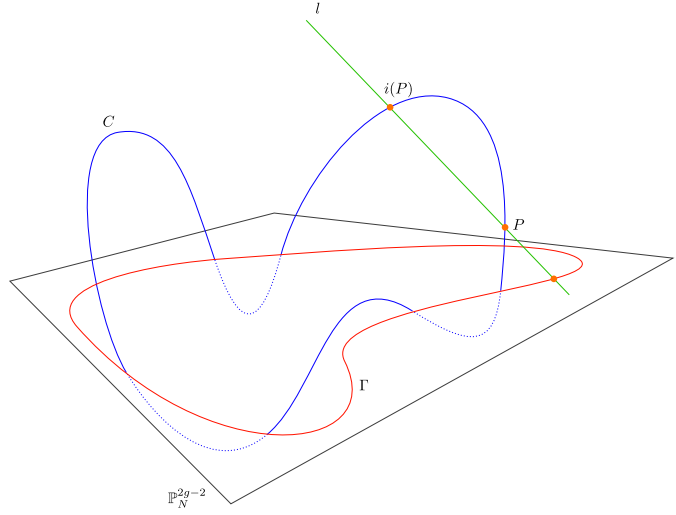
	\caption{The situation in genus $3$. The curves $\Gamma$ and $C$ intersect along the divisor $D$, of degree 6. The secant lines $l$ cutting out the hyperelliptic pencil define the curve $\Gamma$.}\label{fig:gamma}
\end{figure}

\begin{lemma} \label{lem:gamma}
  The locus $\Gamma \subset \PN$ is a non-degenerate rational normal curve in $\PN$. Moreover, $\Gamma$ passes through the $2g$ points $N \subset C$. 
\end{lemma}

\begin{proof}
  Let us start by showing that the intersection $Q_p$ is non-empty for every line $l=\overline{p,i(p)}$, with $p\in C$. Since $\dim |p + i(p)| = \dim |h| = 1$, the intersection $l \cap \PN$ is non-empty by Lemma \ref{lemma:nonempty_intersection}. 

  Let us show that this intersection is a point, i.e. that the line $l$ is not contained in $\PN$.
  Recall that $\PD = |2D + K|^*$.
  If the points $p$ and $i(p)$ are both not contained in the divisor $N$, the vector space $$ V := H^0(C,2D + K - N - (p + i(p)) = H^0(C,2D + K - N - h)$$ is exactly the annihilator of the span $\langle l, \PN \rangle$ in $\PD$. In particular, the codimension of $\langle l, \PN \rangle$ in $\PD$ is the dimension of $V$. By Riemann-Roch and Serre duality, we get that $\dim V = g - 2$, thus $\dim \langle l, \PN \rangle = 3g - 2 - (g - 2) = 2g$. This means that the intersection  $l \cap \PN$ is a point. 

  For the case $p \in N$ and $i(p) \not \in N$, let us remark that the the annihilator of the span $\langle l, \PN \rangle$ is now the vector space $H^0(C,2D + K - N - i(p))$. Since $$h^0(C,2D + K - N - i(p)) < h^0(C,2D + K - N),$$ we conclude that the line $l$ is not contained in $\PN$.
  The case $\{p, i(p) \} \subset N$ is excluded by our genericity hypotheses on $N$.  
Hence we deduce that the locus $\Gamma$ is a curve in $\PN$.

  Let $q$ be a point of $N$. Then, $q$ is a point of $\PN$. Consequently, the line passing through $q$ and $i(q)$ intersects the plane $\PN$ at $q$. Thus, we have that $\Gamma$ passes through the points of $N$.
  Moreover, it is clear that $N$ is the only intersection of $\Gamma$ and $C$, i.e. $\Gamma \cap C = N$.

  Let us prove now that $\Gamma$ is a rational normal curve. Since $\Gamma$ is defined by the hyperelliptic pencil, it is straightforward to see that $\Gamma$ is rational. Moreover, since the divisor $D$ is general, the span of any subset of $2g - 1$ points of $D$ is $\PN$. Thus, it suffices to show that the degree of $\Gamma \subset \PN$ is precisely $2g - 2$.
  
  Let us set $N = q_1 + \cdots + q_{2g}$ with $q_1, \ldots, q_{2g} \in C$.
  By the previous paragraph, $\Gamma$ passes through these $2g$ points. Let us consider a hyperplane $H$ of $\PN$ spanned by $2g - 2$ points of $N$. Without loss of generality, we can suppose that these points are the first $2g - 2$ points $q_1, \ldots, q_{2g - 2}$. To show that the degree of $\Gamma$ is $2g - 2$, we have to show that the intersection of $\Gamma$ with $H$ consists exactly only of these points.

  Let $l$ be the secant line $\overline{q,i(q)}$, $q\in C$. The intersection $l \cap H$ is empty if and only if the linear span $\langle l, H \rangle$ of $l$ and $H$ in $\PD$ is of maximal dimension $2g - 1$, i.e. of codimension $g - 1$ in $\PD$. 
  Consider the divisors 

\begin{align*}
  D_H= q_1 + \cdots + q_{2g - 2} \qquad \mbox{and} \qquad D_l = q + i(q) \ .
\end{align*}

  As before, if $\{q, i(q)\} \cap \{q_1, \ldots, q_{2g - 2} \}$ is empty, the vector space $W = H^0(C,2D + K - D_H - D_l)$ is the annihilator of the span $\langle l, H \rangle$ in $\PD$.
  In particular, the codimension of $\langle l, H \rangle$ in $\PD$ is given by the dimension of $W$. Again by Riemann-Roch and Serre duality theorems, we can check that  

  $$\dim W = h^0(C,-2D + D_H + D_l) + g - 1.$$

  Thus, the codimension of $\langle l, H \rangle$ in $\PD$ is greater than $g - 1$ if and only if $h^0(C,-2D + D_H + D_l) > 0$. Since $\deg (-2D + D_H + D_l) = 0$, this is equivalent to $-2D + D_H + D_l \sim 0$. Since $N = q_1 + \cdots + q_{2g} \sim 2D$, we have that 

\begin{align*}
  -2D + D_H + D_l \sim 0 &\iff q + i(q) \sim q_{2g - 1} + q_{2g} \\ &\iff h \sim q_{2g - 1} + q_{2g} \\ &\iff i(q_{2g - 1}) = q_{2g}.
\end{align*}

  By our genericity hypothesis on $N$, the last condition is not satisfied.
  Consequently, we conclude that the line $l$ intersects the hyperplane $H$ iff $\{q, i(q)\} \cap \{q_1, \ldots, q_{2g - 2} \}$ is non-empty, i.e. iff $q$ or $i(q)$ is one of the $q_k$ for $k = 1, \ldots, 2g - 2$. In particular, $$\Gamma \cap H = \{q_1, \ldots, q_{2g - 2} \}$$
as we wanted to show.
\end{proof}

  Hence, the curve $\Gamma$ is contracted by the map $h_N$ to a point $w \in \MGIT$ by Proposition \ref{prop:RNC}. The point $w$ represents a hyperelliptic curve $C_{w}$ of genus $g-1$ together with an ordering of the Weierstrass points that correspond to the points of $N$ on the rational curve $\Gamma$. 

{}\subsection{The restriction of the theta map to $\MGIT$} \label{sec:general_case}
  Let us set once again $N = p_1 + \cdots + p_{2g}$, a general divisor in the linear system $|2D|$, and consider the span $\PN$ in $\PD$ of the $2g$ marked points $p_1, \ldots,  p_{2g}$. 

  In this Section, we describe birationally the restrictions of $\theta$ to the fibers $p_D^{-1}(N)$ by means of the maps presented in Section \ref{sec:restriction_map}.
  
\subsubsection{The global factorization} \label{sec:Stein_factorization}
  Recall that the base locus of the map $\p$ is the secant variety $\Sec^{g-2}(C)$ by Proposition \ref{prop:linear_systems_coincide}.
As in \cite{bertram}, one can construct a resolution $\widetilde{\p}$ of the map $\p$ via sequence of blow-ups

\begin{equation*} 
  \begin{tikzcd}
    \widetilde{\PD} \arrow[d, "\blowup_{g-1}", swap]   \arrow[rdd, "\widetilde{\p}"] \\
    \vdots \arrow[d, "\blowup_1", swap] \\
    \PD  \arrow[r, dashed, "\p"] & {|}2 \Theta {|}
  \end{tikzcd}
\end{equation*} 

along the secant varieties $$C = \Sec^0(C) \subset \Sec^1(C) \subset \cdots \subset \Sec^{g-1}(C) \subset \PD.$$
  This chain of morphisms is defined inductively as follows:
  the center of the first blow-up $\blowup_1$ is the curve $C = \Sec^0(C)$. 
  For $k = 2, \ldots, g-1$, the center of the blow-up $\blowup_k$ is the strict transform of the secant variety $\Sec^{k-1}(C)$ under the blow-up $\blowup_{k-1}$.
  
  The map $\p$ is, by definition, the composition of the classifying map $f_D$ and the degree 2 map $\theta$. Thus, the map $f_D$ lifts to a map $\widetilde{f_D}$ which makes the following diagram commute:
\begin{equation} \label{diag:global}
  \begin{tikzcd}
    \widetilde{\PD} \arrow[r,"\widetilde{f_D}"] \arrow[rd,swap,"\widetilde{\p}"] & \SU \arrow[d,"\theta"] \\
    &                               {|}2 \Theta {|}            
  \end{tikzcd}
\end{equation}

\subsubsection{Osculating projections}
  We recall here a generalization of linear projections that will allow us to describe the map $p$ in higher genus. For a more complete reference, see for example \cite{massarenti_rischter}.
  Let $X \subset \P^N$ be an integral projective variety of dimension $n$, and $p \in X$ a smooth point. Let
\begin{align*}
  \phi: \mathcal{U} \subset \mathbb{C}^n &\longrightarrow \mathbb{C}^N \\
  (t_1, \ldots, t_n) &\longmapsto \phi(t_1, \ldots, t_n)
\end{align*}
be a local parametrization of $X$ in a neighborhood of $p = \phi(0) \in X$. For $m \geq 0$, let $O^m_p$ be the affine subspace of $\mathbb{C}^N$ passing through $p \in X$ and generated by the vectors $\phi_I(0)$, where $\phi_I$ is a partial derivative of $\phi$ of order $\leq m$.  

  By definition, the \emph{$m$-osculating space} $T_p^m X$ of $X$ at $p$ is the projective closure in $\P^N$ of $O^m_p$. 
  The \emph{$m$-osculating projection} $$\Pi_p^m:X \subset \P^N \dashrightarrow \P^{N_m}$$ is the corresponding linear projection with center $T_p^m$.

\subsubsection{Osculating projections of $\MGIT$} \label{sec:furtherbaselocus} 

In this section we show how the map $\pPN$ induces an osculating projection on $\MGIT$.


  


\begin{lemma} \label{lemma:vanishing_forms}
  Let $Q$ be a $r$-form in $\P^n$ vanishing at the points $P_1$ and $P_2$ with multiplicity $l_1$ and $l_2$ respectively. Then, $Q$ vanishes on the line passing through $P_1$ and $P_2$ with multiplicity at least $l_1 + l_2 - r$.
\end{lemma}
\begin{proof}
  See, for example, \cite[page 2]{kumar_linearsystems}.
\end{proof}
 
Let us now consider the linear system $|\I_{\Sec^N}(g)|$ on $\pPN$ (see Section \ref{sec:restriction_map}). 
The forms in $|\I_{\Sec^N}(g)|$ vanish with multiplicity $g - 1$ along the points of $C$ (see Lemma \ref{prop:linear_systems_coincide}). By Lemma \ref{lemma:vanishing_forms}, these forms vanish then with multiplicity $(g-1) + (g-1) - g = g-2$ along the secant lines $l$ cutting out the hyperelliptic pencil. Thus, these forms vanish with multiplicity $g-2$ on the curve $\Gamma$. 

  Let us also consider the linear system $|\I_{\Sec^{g-2}(N)}(g)|$. Let $\I(\Gamma) \subset |\I_{\Sec^{g-2}(N)}(g)|$ be the partial linear system of forms vanishing (with multiplicity 1) along $\Sec^{g-2}(N)$, and with multiplicity $g - 2$ along $\Gamma$. By our previous observation and Lemma \ref{lem:base_locus}, we have the following inclusions of linear systems:
  
\begin{align*}
\Rlin \subset  |\I_{\Sec^N}(g)| \subset \I(\Gamma) \subset |\I_{\Sec^{g-2}(N)}(g)|.
\end{align*}

These inclusions yield a factorization 

\begin{equation} \label{diag:Kumar}
  \begin{tikzcd}
    \PN \arrow[r,dashed,"h_N"] \arrow[rrd,swap,dashed,"\pPN"] & \MGIT \subset |\I_{\Sec^{g-2}(N)}(g)|^*  \arrow[r,dashed,"\pi_N"] &  {|\I(\Gamma)|}^* \arrow[d,dashed,"l_N"] \\
    &                       &                    {\theta(\SU)}
  \end{tikzcd}
\end{equation} 

  The first map $h_N$ is the one defined in Section \ref{sec:h_N}, its image is the GIT quotient $\MGIT$.
  According to Proposition \ref{prop:RNC}, this map contracts the curve $\Gamma$ to a point $h_N(\Gamma)$. 

\begin{proposition} \label{prop:fact_restringida}
  The map $\pi_N$ is the $(g-3)$-osculating projection $\Pi^{g-3}_{P}$ with center the point $w = h_N(\Gamma)$.
\end{proposition}

\begin{proof} 
  From the geometric description of the linear systems $\I(\Gamma)$ and $|\I_{\Sec^{g-2}(N)}(g)|$ (Prop. \ref{prop:h_N_is_composition} and \ref{prop:RNC}), the base locus of the map $\pi_N$ is the point $w = h_N(\Gamma)$, with multiplicity $g-2$. In particular, since the forms in $\I(\Gamma)$ vanish with multiplicity $g - 2$ along $\Gamma$, the order the projection $\pi_N$ is $g-3$. 
\end{proof}

 According to Proposition \ref{prop:RNC}, the map $h_N$ contracts the curve $\Gamma$ to a point $w$ in $\MGIT$ representing an ordered configuration of the $2g$ marked points $N$. This point in turn corresponds to a hyperelliptic genus $(g - 1)$ curve $C_w$ together with an ordering of the Weierstrass points. Now recall from Sec. \ref{sec:h_N} that the bottom composed map of Diag. \ref{gitcontraction} is the map $h_N$. The rational normal curve $\Gamma\subset \PN$ is contracted to a point $e_0\in \P^{2g-3}$ s.t. $w=i_\Omega(e_0)$. Recall, once again from Sec. \ref{sec:h_N} that $\P^{2g-3}$ also contains the $2g-1$ points $\tau_k\circ \Cr_k(H_i)$, imagese of the hyperplanes $H_i \subset \PN$, with $i\neq k$ and $1\leq i \leq 2g$. Let us lable them $e_1, \dots, e_{2g-1}$.
  Let now $\Lambda$ be the partial linear system of $\Omega$ consisting of the $(g - 1)$-forms in $\P^{2g-3}$ vanishing with multiplicity $g - 2$  in all the points $e_0, e_1, \ldots, e_{2g - 1}$. This linear sub-system induces an osculating linear projection $\kappa: \MGIT \tto  \Lambda ^*$, as seen in Thm. \ref{th:kumar_double_cover}.

\begin{theorem}  \label{th:pi=kappa}
  The map $\pi_N$ coincides with the map $\kappa$. In particular, the map $\pi_N$ is of degree 2.
\end{theorem}

\begin{proof}
Consider the GIT quotient $\MGIT$ embedded in $|\Omega|^*$ as we have seen in Thm. \ref{th:kumar_Omega}. The osculating projection $\pi_N$ is given by the linear system $|\O_{\MGIT}(1) - (g - 2) w|$ of hyperplanes vanishing in $w$ with multiplicity $g-2$. By definition of $\Omega$, this linear system pulls back via $i_{\Omega}$ to the linear system of $(g - 1)$-forms in $\P^{2g - 3}$ vanishing with multiplicity $g - 2$ in $e_1, \ldots, e_{2g - 1}$, and also with multiplicity $g - 2$ in $e_0$, which is precisely $\Lambda$. Hence, the map $\pi_N$ is the map induced by the same linear system as $\kappa$ (see Thm \ref{th:kumar_double_cover}).  
\end{proof}

We will show in the next Section that the map $l_N$ from Diag \ref{diag:Kumar} is actually birational,
and that the map $\pi_N$ coincides with the restriction of the map $\theta$.

\subsection{The hyperelliptic theta map and Rational involutions on $\MGIT$ and $\SU$ } \label{sec:the_global_description}


  The resolution $\widetilde{\p}$ of $\p$ factors through the degree 2 map $\theta$ as shown in Diagram \ref{diag:global}. In the preceding Section we have shown that, when we restrict $\pPN$ to $\PN$, it factors through the degree 2 map $\pi_N$. Now we link these two factorizations. The identification of maps in the following claim must be intended as rational maps, since for example $\pi_N$ is not everywhere defined.
  
\begin{theorem}
  Let $N \in |2D|$ be a general effective divisor. Then, the restricted map $\theta|_{f_D\left(\PN\right)}$ is the map $\pi_N$ up to composition with a birational map. 
\end{theorem}

\begin{proof}
 Let us place ourselves on the open set $\genSU \subset \SU$ of general stable bundles. First we remarl that the factorization $\widetilde{\p} =\theta \circ \widetilde{f_D}$ of Diagram \ref{diag:global} is the Stein factorization of the map $\widetilde{\p}$ along $\widetilde{\PD}$. 
  Indeed, the map $\theta$ is of degree 2 as explained in Section \ref{sec:introduction}. Moreover, the preimage of a general stable bundle $E$ by the map $f_D$ is the $\P^1$ arising as the projectivisation of the space of extensions of the form
  
\begin{align*}
  \quad  0 \to \O(-D)) \to E \to \O(D) \to 0.
\end{align*}
 
  In particular, the fibers of $\widetilde{f_D}$ over $\genSU$ are connected.  

  The restriction of $\p$ to $\PN$ factors through the maps $h_N$ and $\pi_N$ (see Diagrams \ref{diag:Kumar}), followed by the map $l_N$. 
  According to Proposition \ref{prop:RNC}, the fibers of $h_N$ are rational normal curves, thus connected. Moreover, the map $\pi_N$ is degree 2 by Theorem \ref{th:pi=kappa}. 
  By unicity of the Stein factorization, we have our result. 

  Comparing with the factorization  $\widetilde{\p} =\theta \circ \widetilde{f_D}$, we see that $l_N$ cannot have relative dimension $> 0$. Hence, $l_N$ is a finite map. Since the degree of the map $\theta$ in the Stein factorization is 2, which is equal to the degree of $\pi_N$, we have that $l_N$ cannot have degree $> 1$.
  In particular, we have that the map $l_N$ is a birational map. 
\end{proof}

From this description, the arguments of Section \ref{sec:furtherbaselocus} and Thm. \ref{th:kumar_double_cover} we obtain the following


\begin{theorem} \label{th:fibration_kummer}
  The restriction of $\theta$ to the general fiber of $p_{\mathbb{P}_c}$ ramifies on the Kummer variety of dimension $g - 1$, obtained from the Jacobian of the hyperelliptic curve that is the double cover of $\mathbb{P}^1$ ramified along the $2g$ points represented by $P=h_N(\Gamma)$.
\end{theorem}    

\begin{corollary}
One of the irreducible components of the ramification locus of the theta map is birational to a fibration in Kummer $(g-1)$-folds over $\P^g$.
\end{corollary}

Results from \cite[App. E]{hacen} imply that the ramification locus is in fact non-irreducible.

{}\section{The case $g = 3$} \label{sec:g3}
  
  Let us now illustrate the geometric situation by explaining in detail the first case in low genus. We will often tacitly assume that when we say \em map \rm we mean a rational map.
  Let $C$ be a hyperelliptic curve of genus 3. 
  In this setting, we have that the map $\theta$ factors through the involution $i^*$, and embeds the quotient $\SU / {\langle i^* \rangle }$ in $\P^7 = |2 \Theta|$ as a quadric hypersurface (see \cite{beauville_rang2} and \cite{desale_ramanan}).
  Let $D$ be a general effective divisor of degree 3. 
  The projective space $\P^7_D$, as defined in Section \ref{sec:introduction},  parametrizes the extension classes in $\Ext^1(\O(D), \O(-D))$.
  The classifying map $\p$ is given in this case by the complete linear system $|\I_C^2(3)|$ of cubics vanishing on $C$ with multiplicity 2. 
 The forms from this linear system vanish along the secant lines of $C$, and in particular along the secant lines passing through involution-conjugate points. These form a pencil parametrized by the linear system $|h|$.

  The image of the projection of $\theta(\SU)$ with center $\P_c = \P^3 \subset |2 \Theta|$ is also a $\P^3$, that is identified with $|2 D|$ by Theorem \ref{th:proyeccion}. Let $N \in |2D|$ be a general reduced divisor. By Proposition \ref{prop:fibre=image}, the closure of the fiber $\pproj^{-1}(N)$ is the image via $\p$ of the $\PcuatroN$ spanned by the six points of $N$. 

\subsection{The restriction to $\P_N^4$} \label{sec:factor_g3}
  The base locus of the restricted map $\pPN = \pPcuatroN$ contains $\Sec^N = \Sec^1(C) \cap \P^4_N$ by Lemma \ref{lem:base_locus}. 
  The secant variety $\Sec^1(N) \subset \Sec^N$ is the union of the 15 lines passing through pairs of the 6 points of $N$. According to Lemma \ref{lemma:nonempty_intersection}, the further base locus $\Sec^N \setminus \Sec^1(N)$ is given by the intersections of $\P^4_N$ with the lines spanned by degree 2 divisors $L_2$ on $C$ not contained in $N$ satisfying $\dim|L_2| \geq 1$. By Brill-Noether theory, there is only one linear system of such divisors on a genus 3 curve, namely the hyperelliptic linear system $|h|$ (see, for example, \cite{arbarello_cornalba}, Chapter V). We will review these ideas in Section \ref{sec:lowgenera}. This linear system defines, by the intersections with $\PN$ of the lines spanned by the hyperelliptic pencil, the curve $\Gamma$ that we introduced in Section \ref{sec:gamma}.
  Hence, we have that $\Sec^N = \{15 \text{ lines} \} \cup \Gamma$, and the restricted map $\pPN$ factors as

\begin{equation*}
  \begin{tikzcd}
    \P^4_N \arrow[r,dashed,"h_N"] \arrow[rd,swap,dashed,"\pPN"] & \MGITsix \subset \P^4 \arrow[d,dashed,"p"] \\
    &                                           \P^{3}
  \end{tikzcd}
\end{equation*}

where $h_N$ is the map defined by the complete linear system $|\I_{\Sec^1(N)}(3)|$ of cubics vanishing along the 15 lines defined by the points of $N$, and $p$ is the projection with center the image via $h_N$ of the rational normal curve $\Gamma$.

The image of $\pPN$ is a $\P^3$. Indeed, this image cannot have higher dimension, since the map factors through the projection from a point of $\MGITsix \subset \P^4$. Also, it cannot have dimension strictly smaller than 3 since otherwise the relative dimension of $\pPN$ would be bigger than 1, or equivalently the global map $\varphi_D$ would not surjet
onto $\SU$. Hence, in this case the map $\pPN$ is defined exactly by the
system of cubics in $\P^4_N$ vanishing on $\Sec^N$.

  According to Proposition \ref{prop:RNC}, the image of $h_N$ is the GIT moduli space $\MGITsix$ if $N$ is general and reduced. It is a classical result that this GIT quotient is embedded in $\P^4$ as the Segre cubic $S_3$ (see for instance \cite{dolgachev_ortland}). This 3-fold arises by considering the linear system of quadrics in $\P^3$ that pass through five points in general position, thus it is isomorphic to the blow-up of $\P^3$ at these points, followed by the blow-down of all lines joining any two points. The composition of this map with the projection off a smooth point of $S_3$ gives a $2:1$ rational map $\P^3 \dashrightarrow \P^3$ whose ramification locus is a Weddle surface (\cite{kumar,bolowed}).
  The curve $\Gamma \subset \P^4_N$ is a rational normal curve by Lemma \ref{lem:gamma}, hence $\Gamma$ is contracted to a point $P$ by $h_N$ again by Proposition \ref{prop:RNC}.

  By \cite{bertram} and Lemma \ref{lem:base_locus}, the linear system $|\OO_{S_3}(1)|$ of hyperplanes in $S_3$ is pulled back by $h_N$ to $|\I_{\Sec^1(N)}(3)|$ on $\PcuatroN$. The linear system $|\OO_{S_3}(1) - P|$ of hyperplanes in $S_3$ passing through $P$ is pulled back to the complete linear system $|\I_{\Sec^N}(3)|$ defining $\pPN$. 
  Hence, the map $p$ is the linear projection with center $P$.
  Since $S_3$ is a cubic, the projection $p$ is a degree 2 map. We will see in the next Section that this will be also the case for higher genus. The following proposition resumes what we have seen so far in this Section.

\begin{proposition}
  Let $C$ be a hyperelliptic curve of genus 3. Then, for generic $N$, the restriction of $\pD$ to the subspace $\PN$ is exactly the composition $\kappa \circ h_N$.
\end{proposition}

  The point $P$ in $\MGITsix$ represents a rational curve with 6 marked points.
  Let $C'$ be the hyperelliptic genus 2 curve constructed as the double cover of this rational curve ramified in these 6 points.
  According to Theorem 4.2 of \cite{kumar}, the Kummer variety $\Kum(C')$ is contained in the image of $p$, and it is precisely the ramification locus of $\pi$. Recall that, when $g=3$ , the linear system $|2D|$ is a $\P^3$. 
  By Proposition \ref{prop:fibre=image}, the image of $\PcuatroN$ by $\p$ is the closure of the fiber $\pproj^{-1}(N)$. 
  For each point $N$ in $|2D|$, this image is $\P^3 = |\I^2_{\Sec^N}(3)|^*$, which is the image of the Segre variety $\MGITsix$ under the projection with center $P$. Thus, the image of the global map $\p$ is birational to a $\P^3$-bundle over $|2D| = \P^3$. Of course this is also the case since the image of the theta map is a quadric hypersurface in $\P^7$ \cite{desale_ramanan}.
  
{}\section{Explicit descriptions in low genera} \label{sec:lowgenera}

In this Section we will go through an explicit description of the classifying maps and how they factor through forgetful linear systems and osculating projections, for low values of the genus $g(C)$ of the hyperelliptic curve.
In these cases the map remain fairly simple. These computations seem completely out of reach without the help of a computer for higher genus.

\medskip

 Recall from Section \ref{sec:restriction_map} that the intersection $ \Sec^N = \Sec^{g-2}(C) \cap \PN$ arises naturally as part of the base locus of the restricted map $\pPN$. 
  The subvarieties $\Sec^{g-2}(N)$ and $\Gamma$ of $\Sec^N$ yield the factorization of $\pPN$ through the maps $h_N$ and $\pi_N$ of Proposition \ref{prop:fact_restringida}.
  Let us now describe the set 

\begin{align*}
  {\Sec^N}' = \Sec^N \setminus \{ \Gamma \cup \Sec^{g-2}(N) \}.
\end{align*}

  This set is empty for $g = 3$, and the map $\pPN$ is exactly the composition of $h_N$ and $\pi_N$, as described in Section \ref{sec:g3}.
  In higher genus, the existence of non-empty additional base locus $\Sec^N$ corresponds to the fact that the map $\pPN$ may not be exactly the composition of the maps $h_N$ and $\pi_N$. In other words, the map $l_N$ from Diag. \ref{diag:Kumar} may not be non-trivial in higher genus.

  This supplementary base locus is given by the intersections of $\PN$ with $(g-2)$-dimensional $(g-1)$-secant planes of $C$ in $\PD$, which are not already supported on $\Sec^{g-2}(N)$ and $\Gamma$. According to Lemma \ref{lemma:nonempty_intersection}, these intersections are given by effective divisors $L_{g-1}$ on $C$ of degree $g-1$, not contained in $\PN$, and satisfying $\dim |L_{g-1}| \geq 1$. Also by Lemma \ref{lemma:nonempty_intersection}, we obtain $\dim (\langle L_{g-1} \rangle \cap \PN)=\dim |L_{g-1}| - 1$.

  We will now give account of the situation in low genera. 

\subsubsection*{Case $g = 4$} 
  In this case, the divisor $N$ is of degree 8 and the map
\begin{align*}
  \p|_N : \P_N^6 \subset \P_D^{10} \tto |2 \Theta| = \P^{15}
\end{align*} 
is given by the linear system $|\I_C^3(4)|$. This map factors through the map $\pi_N$ which coincides with the $1$-osculating projection $\Pi^1_w$, where $w = h_N(\Gamma)$.
  
  We are looking for degree 3 divisors $L_3$ with $\dim |L_3| \geq 1$. These satisfy all $\dim |L_3| = 1$ and are of the form 
\begin{align*}
  L_3 = h + q \qquad \text{for } q \in C,
\end{align*}
where $h$ is the hyperelliptic divisor. Let $p$ be a point of $C$. Then $L_3 = p + i(p) + q$. Since $\dim |L_3| = 1$, the secant plane $\P^2_{L_3}$ in $\P_D^{10}$ spanned by $p$, $i(p)$ and $q$ intersects $\P^6_N$ in a point. But this point necessarily lies on $\Gamma$, since the line passing through $p$ and $i(p)$ is already contained in this plane. Hence, we do not obtain any additional locus.

\subsubsection*{Case $g = 5$}
  In this case, the divisors $L_4$ of degree 4 are all of the form 
\begin{align*}
  L_4 = h + q + r \qquad \text{for } q, r \in C,
\end{align*}
and satisfy $\dim |L_4| = 1$. Thus, the corresponding secant $\P^3_{L_4}$ spanned by $p$, $i(p)$, $q$ and $r$ intersects $\P^{8}_N$ in a point. As before, this point lies on $\Gamma$, thus we do not obtain any additional locus.
The upshot is the following

\begin{proposition}
  Let $C$ be a hyperelliptic curve of genus 4 or 5, then $\pD$ is defined by a (possibly equal) linear subsystem of the one defining $\kappa \circ h_N$, and set-theoretically the base loci of the two linear systems coincide.
\end{proposition}

\subsubsection*{Case $g = 6$}
  Here we have, as in the genus 5 case, the divisors of the form 
  
\begin{align*}
  L_3 = h + q \qquad \text{for } q \in C,
\end{align*}

which do not give rise to any additional base locus. But there is a new family of divisors 

\begin{align*}
  L_5 = 2h + r \qquad \text{for } r \in C.
\end{align*}

  These divisors satisfy $\dim |L_5| = 2$. In particular, the intersection of the $\P^4_{L_5}$ spanned by $p$, $i(p)$, $q$, $i(q)$ and $r$, for $p, q \in C$, with $\P^{10}_N$ is a line $m$ in $\P^{10}_N$. The line $l_1$ (resp. $l_2$) spanned by $p$ and $i(p)$ (resp. $q$, $i(q)$) intersects $\Gamma$ in a point $\widetilde{p}$ (resp. $\widetilde{q}$). In particular, the line $m$ is secant to $\Gamma$ and passes through  $\widetilde{p}$ and $\widetilde{q}$. Since every point of $\Gamma$ comes as an intersection of a secant line in $C$ with $\P^{10}_N$, we obtain the following description of the base locus of $\pPN$:
  
\begin{proposition}
  Let $C$ be a curve of genus $g = 6$. Then, the base locus of the restricted map $\pPN$ contains the ruled 3-fold $\Sec^1(\Gamma)$.
\end{proposition}

\bibliography{bib_tocho}

\begin{thebibliography}{10}

\bibitem{alzati_bolognesi}
A.~Alzati and M.~Bolognesi.
\newblock A structure theorem for {${\mathcal{SU}}_C(2)$} and the moduli of
  pointed rational curves.
\newblock {\em J. Algebraic Geom.}, 24(2):283--310, 2015.

\bibitem{arbarello_cornalba}
E.~Arbarello, M.~Cornalba, P.~A. Griffiths, and J.~Harris.
\newblock {\em Geometry of algebraic curves}.
\newblock Number v. 1 in Grundlehren der mathematischen Wissenschaften.
  Springer-Verlag, 1985.

\bibitem{beauville_rang2}
A.~Beauville.
\newblock Fibr\'es de rang {$2$} sur une courbe, fibr\'e d\'eterminant et
  fonctions th\^eta.
\newblock {\em Bull. Soc. Math. France}, 116(4):431--448 (1989), 1988.

\bibitem{beauville_vectorbundles}
A.~Beauville.
\newblock Vector bundles on curves and theta functions.
\newblock In {\em Moduli spaces and arithmetic geometry}, volume~45 of {\em
  Adv. Stud. Pure Math.}, pages 145--156. Math. Soc. Japan, Tokyo, 2006.

\bibitem{beauville_narasimhan_ramanan}
A.~Beauville, M.~S. Narasimhan, and S.~Ramanan.
\newblock Spectral curves and the generalised theta divisor.
\newblock {\em J. Reine Angew. Math.}, 398:169--179, 1989.

\bibitem{bertram}
A.~Bertram.
\newblock Moduli of rank-{$2$} vector bundles, theta divisors, and the geometry
  of curves in projective space.
\newblock {\em J. Differential Geom.}, 35(2):429--469, 1992.

\bibitem{bolowed}
M.~Bolognesi.
\newblock On {W}eddle surfaces and their moduli.
\newblock {\em Adv. Geom.}, 7(1):113--144, 2007.

\bibitem{bolognesi_conicbundle}
M.~Bolognesi.
\newblock A conic bundle degenerating on the {K}ummer surface.
\newblock {\em Math. Z.}, 261(1):149--168, 2009.

\bibitem{bolognesi_BLMS}
M.~Bolognesi.
\newblock Forgetful linear systems on the projective space and rational normal
  curves over {$\mathcal{M}^{\operatorname{GIT}}_{0,2n}$}.
\newblock {\em Bull. Lond. Math. Soc.}, 43(3):583--596, 2011.

\bibitem{bolognesi_brivio}
M.~Bolognesi and S.~Brivio.
\newblock Coherent systems and modular subvarieties of {$\mathcal{SU}_C(r)$}.
\newblock {\em Internat. J. Math.}, 23(4):1250037, 23, 2012.

\bibitem{brivio_verra}
S.~Brivio and A.~Verra.
\newblock The theta divisor of {${\mathcal{SU}}_C(2,2d)^s$} is very ample if
  {$C$} is not hyperelliptic.
\newblock {\em Duke Math. J.}, 82(3):503--552, 1996.

\bibitem{desale_ramanan}
U.~V. Desale and S.~Ramanan.
\newblock Classification of vector bundles of rank {$2$} on hyperelliptic
  curves.
\newblock {\em Invent. Math.}, 38(2):161--185, 1976/77.

\bibitem{dolgachev_ortland}
I.~Dolgachev and D.~Ortland.
\newblock Point sets in projective spaces and theta functions.
\newblock {\em Ast\'erisque}, 165:210 pp. (1989), 1988.

\bibitem{drezet_narasimhan}
J.-M. Drezet and M.~S. Narasimhan.
\newblock Groupe de {P}icard des vari\'et\'es de modules de fibr\'es
  semi-stables sur les courbes alg\'ebriques.
\newblock {\em Invent. Math.}, 97(1):53--94, 1989.

\bibitem{kapranov_chow}
M.~M. Kapranov.
\newblock Chow quotients of {G}rassmannians. {I}.
\newblock In {\em I. {M}. {G}el'fand {S}eminar}, volume~16 of {\em Adv. Soviet
  Math.}, pages 29--110. Amer. Math. Soc., Providence, RI, 1993.

\bibitem{kapranov_veronese}
M.~M. Kapranov.
\newblock Veronese curves and {G}rothendieck-{K}nudsen moduli space {$\overline
  M_{0,n}$}.
\newblock {\em J. Algebraic Geom.}, 2(2):239--262, 1993.

\bibitem{keel_mckernan}
S.~Keel and J.~McKernan.
\newblock Contractible extremal rays on {$\overline M_{0,n}$}.
\newblock In {\em Handbook of moduli. {V}ol. {II}}, volume~25 of {\em Adv.
  Lect. Math. (ALM)}, pages 115--130. Int. Press, Somerville, MA, 2013.

\bibitem{knudsen_projectivity}
F.~F. Knudsen.
\newblock The projectivity of the moduli space of stable curves. {II}. {T}he
  stacks {$M\sb{g,n}$}.
\newblock {\em Math. Scand.}, 52(2):161--199, 1983.

\bibitem{kumar}
C.~Kumar.
\newblock Invariant vector bundles of rank 2 on hyperelliptic curves.
\newblock {\em Michigan Math. J.}, 47(3):575--584, 2000.

\bibitem{kumar_linearsystems}
C.~Kumar.
\newblock Linear systems and quotients of projective space.
\newblock {\em Bull. London Math. Soc.}, 35(2):152--160, 2003.

\bibitem{lange_narasimhan}
H.~Lange and M.~S. Narasimhan.
\newblock Maximal subbundles of rank two vector bundles on curves.
\newblock {\em Math. Ann.}, 266(1):55--72, 1983.

\bibitem{massarenti_rischter}
A.~Massarenti and R.~Rischter.
\newblock Non-secant defectivity via osculating projections, 2016.

\bibitem{narasimhan_ramanan_moduli}
M.~S. Narasimhan and S.~Ramanan.
\newblock Moduli of vector bundles on a compact {R}iemann surface.
\newblock {\em Ann. of Math. (2)}, 89:14--51, 1969.

\bibitem{ortega_coble}
A.~Ortega.
\newblock On the moduli space of rank 3 vector bundles on a genus 2 curve and
  the {C}oble cubic.
\newblock {\em J. Algebraic Geom.}, 14(2):327--356, 2005.

\bibitem{pareschi_popa}
G.~Pareschi and M.~Popa.
\newblock Regularity on abelian varieties. {I}.
\newblock {\em J. Amer. Math. Soc.}, 16(2):285--302, 2003.

\bibitem{pauly_coble}
C.~Pauly.
\newblock Self-duality of {C}oble's quartic hypersurface and applications.
\newblock {\em Michigan Math. J.}, 50(3):551--574, 2002.

\bibitem{raynaud}
M.~Raynaud.
\newblock Sections des fibr\'es vectoriels sur une courbe.
\newblock {\em Bull. Soc. Math. France}, 110(1):103--125, 1982.

\bibitem{vangeemen_izadi}
B.~van Geemen and E.~Izadi.
\newblock The tangent space to the moduli space of vector bundles on a curve
  and the singular locus of the theta divisor of the {J}acobian.
\newblock {\em J. Algebraic Geom.}, 10(1):133--177, 2001.

\bibitem{hacen}
H.~Zelaci.
\newblock Moduli spaces of anti-invariant vector bundles over curves and
  conformal blocks.
\newblock {\em PhD THesis}, pages 1--111, 2017.

\end{thebibliography}
\bibliographystyle{plain}

\end{document}